 \documentclass[12pt]{article}

\usepackage{amsmath,amsfonts,amssymb}
\usepackage{physics}
\usepackage{tikz-cd}
\usepackage{amsthm}

\usepackage{bm}

 \usepackage[numbers,sort&compress]{natbib}
\usepackage{doi}
\renewcommand{\doitext}{\noexpand\textsc{doi:}}

\theoremstyle{definition}
\newtheorem{defi}{Definition} 

\theoremstyle{plain}
\newtheorem{theorem}[defi]{Theorem} 

\newcommand{\R}{\mathbb{R}}

\renewcommand{\S}{\mathbb{S}}

\usepackage{hyperref}
\hypersetup{
   colorlinks   =  true,
    linkcolor    = blue,
    citecolor    = red,
     urlcolor	=blue,     
	urlbordercolor = {1 1 0}
}

\begin{document}

  \vskip0.5cm

  \begin{center}
 
 \noindent
 {\Large \bf  
New Lie systems from Goursat distributions:\\[4pt] reductions and reconstructions} 
   \end{center}

\medskip

\begin{center}

{\sc Oscar Carballal$^{1,2}$}

\end{center}

\medskip

\noindent
$^1$ Departamento de \'Algebra, Geometr\'{\i}a y Topolog\'{\i}a,  Facultad de Ciencias 
Matem\'aticas, Universidad Complutense de Madrid, Plaza de Ciencias 3, E-28040 Madrid, Spain 

\noindent
$^2$ Instituto de Matem\'atica Interdisciplinar, Universidad Complutense de Madrid, E-28040 Madrid,  Spain

 \medskip
 
\noindent  E-mail: {\small  \href{mailto:oscarbal@ucm.es}{\texttt{oscarbal@ucm.es}}}

\begin{abstract} 
\noindent
We show that types of bracket-generating distributions lead to new classes of Lie systems with compatible geometric structures. Specifically, the $n$-trailer system is analysed, showing that its associated distribution is related to a Lie system if $n = 0$ or $n = 1$. These systems allow symmetry reductions and the reconstruction of solutions of the original system from those of the reduced one. The reconstruction procedure is discussed and indicates potential extensions for studying broader classes of differential equations through Lie systems and new types of superposition rules.
\end{abstract}
\medskip
\medskip



\medskip
 
\newpage
  
\tableofcontents

\section{Introduction}
{\it Lie systems} are first-order systems of $t$-dependent ordinary differential equations (ODEs) whose general solution can be expressed through a $t$-independent function, a {\it superposition rule}, in terms of a generic finite family of particular solutions and some constants related to the initial conditions \cite{Lucas2020}. 

Although Lie systems are more of an exception than a rule among ODEs \cite{Lucas2020}, they admit significant mathematical properties and applications in physics. Notably, they can be made compatible with many geometric structures \cite{Lucas2020}. Lie systems on the real line and the plane have been extensively studied and classified under local diffeomorphisms at generic points \cite[Chapter~4]{Lucas2020}. However, those in higher dimensions remain largely unexplored \cite{Campoamor-Stursberg2025a,CampoamorStursberg2025}. In this context, we show that some {\it Goursat distributions}, to which considerable attention has been given in differential geometry, $k$-contact geometry and topology \cite{PasillasLepine2001,Lucas2025,Montgomery2006}, can  give rise to new applications of Lie systems.

Every $t$-dependent system of  ODEs  on a  manifold $M$
of the form 
\begin{equation}
    \dv{x}{t} = X(t,x)
\label{eq:system}
\end{equation}
is univocally associated with a $t$-dependent vector field $X: \mathbb{R} \times M \to \mathrm{T} M$, whose integral curves correspond to the solutions of the system. A superposition rule for \eqref{eq:system} is a $t$-independent map $\Psi: M^{s} \times M \to M$ through which the general solution $x(t)$ of the system can be expressed as
$$
x(t) = \Psi(x_{(1)}(t), \ldots, x_{(s)}(t); \lambda),
$$
where $x_{(1)}(t), \ldots, x_{(s)}(t)$ are generic particular solutions and $\lambda \in M$ is a point related to the initial conditions. The renowned Lie–Scheffers theorem \cite[Section 3.8]{Lucas2020} states that  system \eqref{eq:system} is a Lie system if and only if the smallest Lie algebra $V^X$ containing the vector fields $\{X_t:= X(t, \cdot)\}_{t \in \mathbb{R}}$ on $M$ is finite-dimensional. Any finite-dimensional Lie algebra $V$ such that $\{X_{t}\}_{t \in \R} \subset V$ is called a {\it Vessiot--Guldberg  Lie algebra} (VG Lie algebra, hereafter) of $X$.

A {\it bracket-generating distribution} (also called {\it non-holonomic}) on $M$ is a regular distribution $\mathcal{D} \subset \mathrm{T}M$ whose associated {\it derived flag }
$$
\mathcal{D}^{(0)} \subset \mathcal{D}^{(1)} \subset \cdots \subset \mathcal{D}^{(i)} \subset \cdots \subset \mathrm{TM},
$$
inductively defined by $\mathcal{D}^{(0)}:= \mathcal{D}$ and $\mathcal{D}^{(i+1)} := \mathcal{D}^{(i)} + [\mathcal{D}^{(i)}, \mathcal{D}^{(i)}]$ for $i \geq 0$, satisfies $\mathcal{D}^{(r)} = \mathrm{T}M$ for some $r \geq 1$ (see \cite{Montgomery2006}). A {\it Goursat distribution} on $M$, with $\dim M \geq 3$, is a regular distribution $\mathcal{D} \subset \mathrm{T}M$ fulfilling the {\it Goursat condition} 
$$\mathrm{rk}\: \mathcal{D}^{(i)} = i + 2, \qquad 0 \leq i \leq \dim M -2.$$
Thus, Goursat distributions form a special class of bracket-generating distributions of rank two. If $\dim M = 3$, they correspond to {\it contact distributions}, while if $\dim M = 4$, they are known as {\it Engel distributions} (see \cite[Chapter 6]{Montgomery2006}).

The structure of this paper goes as follows. In Section~\ref{section:trailer}, we examine the {\it $n$-trailer system} \cite{Bravo-Doddoli2015,Montgomery2001,PasillasLepine2001} for the cases $n = 0, 1$, demonstrating that its associated constraint distribution, which is a Goursat distribution \cite{PasillasLepine2001}, gives rise to Lie systems. The resulting systems are shown to be invariant under a principal Lie group action, thereby inducing reduced Lie systems on the corresponding quotient spaces. Furthermore, invariant forms yielding principal connection forms on the associated principal bundles are constructed. More concretely, in the case of the $0$-trailer system,  a contact form is obtained, turning the system into a contact Lie system \cite{Lucas2020}, while for the $1$-trailer system, a novel method for constructing invariant forms is introduced, extending the approach developed in \cite{Gracia2019}.  This ansatz enables the reconstruction of solutions of the original system from those of the reduced one, and the reconstruction procedure is analysed in detail in both cases. Finally,  Section~\ref{section:reconstruction} discusses the reconstruction procedure for Lie systems in the general setting, as established in Theorem~\ref{thm:reconstruction}, and suggests potential extensions of this framework for the study of broader classes of first-order systems of ODEs via Lie systems through a reconstruction procedure. 

\section{The $n$-trailer system} \label{section:trailer}
The $n$-trailer system, studied in differential geometry, mechanics, and control theory (see \cite[Section 3]{PasillasLepine2001}, \cite{Bravo-Doddoli2015}, \cite[Appendix D]{Montgomery2001}), consists of a leading car towing $n \geq 0$ trailers, with the tow hook of each trailer located at the centre of its single axle. This system is subjected to $(n+1)$-nonholonomic constraints. The constraint distribution is a Goursat distribution and, moreover, all possible germs of Goursat distributions of corank $n+1$ are realized by the $n$-trailer system at its different points \cite{Montgomery2001,PasillasLepine2001}.

 Specifically, on $\R^{2} \times (\S^{1})^{n+1}$ with local coordinates $(\xi_{1}, \xi_{2}, \theta_{0}, \ldots, \theta_{n})$, the constraint distribution $\mathcal{D}$ of the $n$-trailer system is spanned by the linearly independent vector fields 
 \begin{equation*}
       X_{1}:=\frac{\partial}{\partial\theta_{n}},\,\,\,\,
       X_{2}:=\pi_{0}\cos(\theta_{0})\frac{\partial}{\partial \xi_{1}}+\pi_{0}\sin(\theta_{0})\frac{\partial}{\partial \xi_{2}}+\sum_{i=0}^{n-1}\pi_{i+1}\sin(\theta_{i+1}-\theta_{i})\frac{\partial}{\partial \theta_{i}},
 \end{equation*}
where $\pi_{i}:=\prod_{j=i+1}^n\cos(\theta_{j}-\theta_{j-1})$ for $0 \leq i \leq n-1$ and $\pi_{n}:=1$. The coordinates $\xi_1$ and $\xi_2$ represent the position of the last trailer, while $\theta_0, \dots, \theta_n$ denote the angles between each trailer’s axle and the $\xi_1$-axis, starting from the last trailer.

Let us now consider the $t$-dependent vector field on $\R^{2} \times (\S^{1})^{n+1}$ given by 
\begin{equation}
X := b_{1}(t) X_{1} + b_{2}(t) X_{2},
    \label{eq:trailer:tdep}
\end{equation}
where $b_{1}(t), b_{2}(t) \in C^{\infty}(\R)$ are linearly independent $t$-dependent functions that can be interpreted as control inputs of the $n$-trailer system. For $n = 2$, it can be demonstrated, albeit through a lengthy yet routine computation, that the Lie algebra $V^{X} = \langle X_{1}, X_{2} \rangle$  is infinite-dimensional. Hence, \eqref{eq:trailer:tdep} does not define a Lie system in this case.  Indeed, the expressions of the elements of $V^{X}$ become progressively more intricate as  $n \geq 2$ increases \cite{Jean1996}. 

Let us  examine \eqref{eq:trailer:tdep} for $n = 0$ and $n = 1$, proving that it is a Lie system in these cases. The case $n = 0$ corresponds to the Chaplygin sleigh system on $\R^{2} \times \S^{1}$ \cite{Chaplygin2008}. In this case, \eqref{eq:trailer:tdep} reads
\begin{equation}
X = b_{1}(t) X_{1} + b_{2}(t) X_{2}= b_{1}(t) \pdv{\theta_{0}} + b_{2}(t) \left( \cos(\theta_{0}) \pdv{\xi_{1}} + \sin(\theta_{0}) \pdv{\xi_{2}} \right),
    \label{eq:t1:tdep}
\end{equation}
and the vector fields $X_{1}$ and $X_{2}$ span a contact distribution on $\R^{2} \times \S^{1}$, as 
$$
X_{3}:= [X_{1}, X_{2}] = - \sin(\theta_{0}) \pdv{\xi_{1}} + \cos(\theta_{0}) \pdv{\xi_{2}}
$$
and $X_{1} \wedge X_{2} \wedge X_{3}$ is non-vanishing. Moreover, 
\begin{equation}
[X_{1}, X_{2}] = X_{3}, \qquad [X_{1}, X_{3}] = - X_{2}, \qquad [X_{2}, X_{3}] = 0,
    \label{eq:iso2}
\end{equation}
and they span a VG Lie algebra $V^{X} \simeq \mathfrak{iso}(2)$. This is a locally automorphic Lie system. That is, the vector fields of $V^{X}$ span $\mathrm{T}(\R^{2} \times \S^{1})$ and $\dim V^{X} = \dim (\R^{2} \times \S^{1})$. Consequently, the Lie algebra of Lie symmetries of $V^{X}$, $\mathrm{Sym}(V^{X})$, formed by all the vector fields $Y \in \mathfrak{X}(\R^{2} \times \S^{1})$ such that $\mathcal{L}_{Y}Z = 0$ for all $Z \in V^{X}$, is isomorphic to $\mathfrak{iso}(2)$ as well  \cite{Gracia2019}. A routine computation shows that 
$$Y_{1} := - \xi_{2} \pdv{\xi_{1}} + \xi_{1} \pdv{\xi_{2}} + \pdv{\theta_{0}}, \qquad Y_2:=\frac{\partial}{\partial \xi_1}, \qquad Y_3:=\frac{\partial}{\partial \xi_2},$$
with opposite commutation rules to \eqref{eq:iso2}, generate $\mathrm{Sym}(V^{X})$. Let now $\{\alpha^{1}, \alpha^{2}, \alpha^{3}\}$ be the dual frame to $\{Y_{1}, Y_{2}, Y_{3}\}$, namely
$$ \alpha_{1} = \dd \theta_{0}, \qquad \alpha_{2} = \dd \xi_{1} + \xi_{2} \dd \theta_{0}, \qquad \alpha_{3} = \dd \xi_{2} - \xi_{1} \dd \theta_{0}.
$$
Clearly, $
\dd \alpha_{1} = 0$, $\dd \alpha_{2} =- \alpha_{1} \wedge \alpha_{3}$, $
\dd \alpha_{3} =  \alpha_{1} \wedge \alpha_{2}$
and  $\alpha_{2}$ and $\alpha_{3}$ are contact forms on $\R^{2} \times \S^{1}$ (that is, they satisfy $\alpha_{j} \wedge \dd \alpha_{j} \neq 0$ for $j = 2, 3$). Moreover, as they are invariant forms for $V^{X}$ \cite{Gracia2019}, they turn $X$ into a contact Lie system of Liouville type concerning each of these contact structures \cite{Lucas2023}. Of course, their associated Reeb vector fields are $Y_{2}$ and $ Y_{3}$, respectively.  
 Remarkably, the Reeb vector field $Y_{3}$ gives rise to an $\mathbb{R}$-principal action $\Phi: \mathbb{R} \times (\mathbb{R}^{2} \times \mathbb{S}^{1}) \to \mathbb{R}^{2} \times \mathbb{S}^{1}$ given by translations in the $\xi_{2}$-axis. Its associated $\mathbb{R}$-principal bundle is the trivial $\mathbb{R}$-bundle
 $ \pi: \mathbb{R}^{2} \times \mathbb{S}^{1} \to \mathbb{R}\times \mathbb{S}^{1}$. The $t$-dependent vector field \eqref{eq:t1:tdep} projects via $\pi$ onto the Lie system with VG Lie algebra $\langle \pi_{*}X_{1}, \pi_{*}X_{2}, \pi_{*}X_{3} \rangle$ given by
\begin{equation}
\pi_{*} X = b_{1}(t) \pdv{\theta_{0}} + b_{2}(t) \cos(\theta_{0}) \pdv{\xi_{1}}.
    \label{eq:1t:proj}
\end{equation}
Solutions of system \eqref{eq:t1:tdep} on \( \mathbb{R}^2 \times \S^1 \) can be obtained from those of the reduced system \eqref{eq:1t:proj} following a reconstruction approach \cite{Marsden1990}. The following demonstrates how this reconstruction problem can be tackled in this context.

Let $\gamma: \mathbb{R} \ni t \mapsto \gamma(t) \in \mathbb{R} \times \mathbb{S}^{1}$ be a generic solution of the reduced Lie system \eqref{eq:1t:proj}. The contact form $\alpha_{3}$ is a principal connection form 
on the bundle $\pi: \mathbb{R}^{2} \times \mathbb{S}^{1} \to \mathbb{R} \times \mathbb{S}^{1}$. Next, consider the system on $\mathbb{R}^{2} \times \mathbb{S}^{1}$ given by
\begin{equation*}
        \widetilde{X}_{t} := X_{t} - \alpha_{3}(X_{t}) Y_{3} = b_{1}(t) \left( \pdv{\theta_{0}} + \xi_{1} \pdv{\xi_{2}} \right) + b_{2}(t) \cos(\theta_{0}) \pdv{\xi_{1}}, \qquad t \in \mathbb{R},
\end{equation*}
which is the horizontal lift of $\pi_{*}X$ with respect to $\alpha_{3}$. That is, $\widetilde{X}_{t} \in \ker \alpha_{3}$ and $\pi_{*} \widetilde{X}_{t} = \pi_{*}X_{t}$ for all $t \in \mathbb{R}$. Consequently, the solutions of $\widetilde{X}$ are the horizontal lifts of the solutions of $\pi_{*}X$ with respect to $\alpha_{3}$. Let us denote by $\widetilde{\gamma}$ the horizontal lift of $\gamma$ starting at a point $x_{0} \in \pi^{-1}(\gamma(0))$, and let $g(t)$ be  any solution of  the Lie system on $\mathbb{R}$ given by 
\begin{equation*}
        \dv{g}{t} = \alpha_{3}(X_{t}(\widetilde{\gamma}(t)))=  - b_{1}(t) \xi_{1}(t) + b_{2}(t)\sin(\theta_{0}(t)).
\end{equation*}
It follows that $\mathbb{R} \ni t  \mapsto x(t):=\Phi(g(t), \widetilde{\gamma}(t)) \in \mathbb{R}^{2} \times \mathbb{S}^{1} $
is a solution of system \eqref{eq:trailer:tdep} on $\R^{2} \times \S^{1}$ with the initial condition $x(0) = \Phi(g(0),x_{0})$ (see Theorem~\ref{thm:reconstruction}). This example closely aligns with the ansatz introduced in \cite{Carinena2001}, though it is now approached from a contact geometric perspective. More importantly, these ideas can be extended to broader scenarios, as shown next.

Let us now study the system \eqref{eq:trailer:tdep} on $\R^{2} \times \mathbb{T}^{2}$  associated with the $1$-trailer system, where $\mathbb{T}^{2} := \S^{1} \times \S^{1}$ is the $2$-torus. In this case, the $t$-dependent vector field $X = b_{1}(t)X_{1} + b_{2}(t) X_{2}$ from  \eqref{eq:trailer:tdep} reads 
\begin{equation}
    \begin{split}
        X &= b_{1}(t) \pdv{\theta_{1}}\\ 
        &+ b_{2}(t) \left(\cos(\theta_1-\theta_0)\left(\cos (\theta_0)\frac{\partial}{\partial \xi_1}+\sin(\theta_0)\frac{\partial}{\partial \xi_2}\right)+\sin(\theta_1-\theta_0)\frac{\partial}{\partial \theta_0} \right),
    \end{split}
    \label{eq:t2:tdep}
\end{equation}
with $X_{1}$ and $X_{2}$  spanning an Engel distribution on $\R^{2} \times \mathbb{T}^{2}$.
These vector fields generate the six-dimensional Lie algebra $V^{X}$, on which the $t$-dependent vector field  \eqref{eq:t2:tdep} takes values, showing that it is a Lie system. A basis for $V^{X}$ consists of $X_{1}, X_{2}$, together with the vector fields
\begin{equation*}
\begin{split}
     &X_{3} := -\sin(\theta_{1}- \theta_{0}) \left( \cos(\theta_{0}) \pdv{\xi_{1}} + \sin(\theta_{0}) \pdv{\xi_{2}} \right) + \cos(\theta_{1}- \theta_{0}) \pdv{\theta_{0}}, \\
      &X_{4} := \sin(\theta_{0}) \pdv{\xi_{1}}  - \cos(\theta_{0}) \pdv{\xi_{2}} + \pdv{\theta_{0}}, \\
      &X_{5} := \cos(\theta_{1} - \theta_{0}) \left( \sin(\theta_{0}) \pdv{\xi_{1}} - \cos(\theta_{0}) \pdv{\xi_{2}} + \pdv{\theta_{0}}\right), \\
      &X_{6} := - \sin(\theta_{1}- \theta_{0}) \left( \sin(\theta_{0}) \pdv{\xi_{1}} - \cos(\theta_{0}) \pdv{\xi_{2}} + \pdv{\theta_{0}}\right).
      \end{split}
\end{equation*}
Their non-zero commutation relations read
\begin{equation*}
    \begin{aligned}
&[X_{1}, X_{2}] = X_{3}, \qquad &&[X_{1}, X_{3}] = - X_{2}, \qquad && [X_{1}, X_{5}] = X_{6}, \\
& [X_{1}, X_{6}] = - X_{5}, \qquad && [X_{2}, X_{3}] = X_{4}, \qquad && [X_{2}, X_{4}] = X_{5}, \\
& [X_{2}, X_{5}] = X_{4}, \qquad && [X_{3}, X_{4}] = X_{6}, \qquad && [X_{3}, X_{6}] = X_{4}, \\
& [X_{4}, X_{5}] = - X_{6}, \qquad && [X_{4}, X_{6}] = X_{5}, \qquad && [X_{5}, X_{6}] = X_{4}.
    \end{aligned}
\end{equation*}
Hence, $V^{X}$ is a decomposable Lie algebra isomorphic to $\mathfrak{sl}(2, \mathbb{R}) \oplus \mathfrak{iso}(2)$.   More concretely, $\langle X_{4}, X_{5}, X_{6} \rangle \simeq \mathfrak{sl}(2, \mathbb{R})$, while $\langle X_{1} + X_{4}, X_{2}+ X_{6}, X_{3}- X_{5} \rangle \simeq \mathfrak{iso}(2)$.  Thus, $X$ is not a locally automorphic Lie system and no general method to derive invariant tensors for $V^{X}$ is known \cite{Gracia2019}, as was previously done for the system associated with the $0$-trailer case. However, taking into account that the subspace $\langle X_{1} + X_{4}, X_{2} + X_{6}, X_{3} - X_{5} \rangle$ forms a Lie subalgebra of $V^{X}$ with linearly independent generators, the problem can be approached using a novel ansatz.

Let us denote $Z_{1} := X_{1}+X_{4}$, $Z_{2}:= X_{2}+X_{6}$ and $Z_{3} := X_{3}-X_{5}$. We have  that $Z_{1} \wedge Z_{2} \wedge Z_{3} \wedge X_{4}$ is non-vanishing. Thus, $\langle Z_{1}, Z_{2}, Z_{3}, X_{4} \rangle \simeq \mathfrak{iso}(2) \oplus \mathbb{R}$ is a locally automorphic VG Lie algebra with the following  commutation rules
\begin{equation}
[Z_{1}, Z_{2}] = Z_{3}, \qquad [Z_{1}, Z_{3}] = -Z_{2}, \qquad [Z_{2}, Z_{3}] = [X_{4}, \cdot] = 0.
    \label{eq:1t:cr}
\end{equation}
Its associated Lie algebra of Lie symmetries is spanned by 
\begin{equation*}
    \begin{split}
        &Y_{1}:= \xi_{1} \pdv{\xi_{2}} - \xi_{2}\pdv{\xi_{1}} + \pdv{\theta_{0}} + \pdv{\theta_{1}}, \qquad Y_{2}:= \pdv{\xi_{1}}, \\
        &Y_{3}:= \pdv{\xi_{2}}, \qquad Y_{4} := \sin(\theta_{0}) \pdv{\xi_{1}} - \cos(\theta_{0}) \pdv{\xi_{2}} + \pdv{\theta_{0}},
    \end{split}
\end{equation*}
with opposite commutation rules to \eqref{eq:1t:cr}. In particular, $Y_{2}$ and $Y_{3}$ are Lie symmetries for $V^{X}$. Consider now the dual frame to $\{Y_{1}, \ldots, Y_{4}\}$ given by  
\begin{equation*}
    \begin{split}
       & \alpha_{1} = \dd \theta_{1} , \qquad \alpha_{2} = \dd \xi_{1} + \xi_{2} \dd \theta_{1} - \sin(\theta_{0})(\dd \theta_{0} - \dd \theta_{1}), \\
       & \alpha_{3} = \dd \xi_{2} + \cos(\theta_{0})(\dd \theta_{0} - \dd \theta_{1}) - \xi_{1} \dd \theta_{1}, \qquad \alpha_{4} = \dd \theta_{0} - \dd \theta_{1}.
    \end{split}
\end{equation*}
Since every invariant form for $V^{X}$ is likewise invariant for $\langle Z_{1}, Z_{2}, Z_{3}, X_{4} \rangle$, any such form must be a linear combination with real coefficients of exterior products of $\alpha_{1}, \ldots, \alpha_{4}$. Among them,  $\alpha_{2}$ and $\alpha_{3}$ are also invariant forms for the elements of $V^{X}$. In other words, our idea gives a method to find invariant forms.  The Lie symmetry $Y_{3}$ gives rise to an $\mathbb{R}$-principal action $\Phi: \mathbb{R} \times (\mathbb{R}^{2} \times \mathbb{T}^{2}) \to \mathbb{R}^{2} \times \mathbb{T}^{2}$ given by translations in the $\xi_{2}$-axis of $\mathbb{R}^{2}$. Its associated  $\mathbb{R}$-principal bundle is the trivial bundle $\pi: \mathbb{R}^{2} \times \mathbb{T}^{2} \to \mathbb{R} \times \mathbb{T}^{2}$, on which  
$\alpha_{3}$ defines a principal connection. 

Let $\gamma: \mathbb{R} \ni t \mapsto \gamma(t) \in \mathbb{R} \times \mathbb{T}^{2}$ be a generic solution of the reduced Lie system $\pi_{*}X = b_{1}(t)X_{1} + b_{2}(t) \pi_{*}X_{2}$, where 
\begin{equation*}
\pi_{*}X_{1} = \pdv{\theta_{1}}, \qquad \pi_{*}X_{2} =  \cos(\theta_{1}- \theta_{0}) \cos(\theta_{0}) \pdv{\xi_{1}}  +\sin(\theta_{1}-\theta_{0}) \pdv{\theta_{0}},
\end{equation*}
and consider its horizontal lift $\widetilde{\gamma}: \mathbb{R} \ni t \mapsto \widetilde{\gamma}(t) \in \mathbb{R}^{2} \times \mathbb{T}^{2}$ with respect to $\alpha_{3}$ starting at point $x_{0} \in \pi^{-1}(\gamma(0))$, which is a solution of the horizontal lift 
$$
\widetilde{X} := b_{1}(t) (X_{1} - \alpha_{3}(X_{1})Y_{3}) +b_{2}(t) (X_{2} - \alpha_{3}(X_{2})Y_{3})
$$
of $\pi_{*}X$ with respect to $\alpha_{3}$.  Next,  we consider the Lie system on  $\R$ given by 
\begin{equation*}
\begin{split}
   &\dv{g}{t} = \alpha_{3}(X(\tilde{\gamma} (t))) = -b_{1}(t) (\xi_{1}(t) + \cos(\theta_{0}(t))) + b_{2}(t) \sin(\theta_{1}(t)),
   \end{split}
\end{equation*}
and let $g(t)$ be any solution of it. Then, $
\R \ni t \mapsto x(t):= \Phi(g(t), \widetilde{\gamma}(t)) \in \R^{2} \times \mathbb{T}^{2} $
is a solution of the Lie system \eqref{eq:trailer:tdep} on $\R^{2} \times \mathbb{T}^{2}$ such that $x(0) = \Phi(g(0),x_{0})$ (see Theorem~\ref{thm:reconstruction}).

\section{Reconstruction of Lie systems} \label{section:reconstruction}
The following shows how to solve the reconstruction problem \cite{Marsden1990} for Lie systems.
\begin{theorem}
\label{thm:reconstruction}
    Let $X$ be a Lie system on a manifold $M$ possessing a VG Lie algebra $V$ formed by $G$-invariant vector fields with respect to a principal Lie group action $\Phi: G \times M \to M$. Let $\bm{\eta}$  be a principal connection form on the associated principal bundle $\pi: M \to M/G$. Suppose the following:
    \begin{itemize}
        \item[(i)] $\gamma: \mathbb{R} \ni t \mapsto \gamma(t) \in M/G$ is a generic solution of the reduced Lie system $\pi_{*}X$ on $M/G$, and let $\widetilde{\gamma}: \mathbb{R} \ni t \mapsto \widetilde{\gamma}(t) \in M/G$ be its  horizontal lift  with respect to $\bm{\eta}$ starting at a point $x_{0} \in \pi^{-1}(\gamma(0))$, which is a solution of the horizontal lift $\widetilde{X}$ of $\pi_{*}X$ with respect to $\bm{\eta}$; and 
        \item[(ii)] $g: \mathbb{R} \ni t \mapsto g(t) \in G$ is any solution of the  Lie system on $G$ given by 
        \begin{equation*}
        \dv{g}{t} = \mathrm{T}_{e}L_{g(t)}\left[ \bm{\eta}(X_{t}(\widetilde{\gamma}(t))) \right].
            \label{eq:systemgroup}
        \end{equation*}
    \end{itemize}
    Then, $\mathbb{R} \ni t \mapsto x(t):= \Phi(g(t),\widetilde{\gamma}(t))$ is a solution of $X$ such that  $x(0) =\Phi(g(0),x_{0})$.
\end{theorem}

\begin{proof}
First of all, denote by $^{\#}: \mathfrak{g} \to \mathfrak{X}(M)$ the anti-homomorphism of Lie algebras mapping every $v \in \mathfrak{g}$ into its infinitesimal generator $v^{\#} \in \mathfrak{X}(M)$ given by $v^{\#}(p):= \dv{t} \vert_{t = 0}\Phi(\exp(tv),p)$ for all $p \in M$.

We have that (see \cite[p. 19]{Marsden1990})
    \begin{equation}
        \dv{x}{t} = \mathrm{T}_{\widetilde{\gamma}(t)} \Phi_{g(t)} \left( \dv{\widetilde{\gamma}}{t} \right) + \mathrm{T}_{\widetilde{\gamma}(t)} \Phi_{g(t)} \left(\mathrm{T}_{g(t)} L_{g(t)^{-1}} \left( \dv{g}{t} \right) \right)^{\#}(\widetilde{\gamma}(t)).
        \label{eq:sol}
    \end{equation}
    On one hand, the horizontal component of \eqref{eq:sol} projects onto 
    \begin{equation}
    \begin{split}
        \mathrm{T}_{x(t)} \pi \left(\dv{x}{t}\right) &= (\mathrm{T}_{x(t)} \pi \circ \mathrm{T}_{\widetilde{\gamma}(t)} \Phi_{g(t)}) \left( \dv{\widetilde{\gamma}}{t} \right) = \mathrm{T}_{\widetilde{\gamma}(t)} \pi (\widetilde{X}_{t}(\widetilde{\gamma}(t))) \\
        & = (\pi_{*}X_{t})(\gamma(t)) = \mathrm{T}_{x(t)} \pi (X_{t}(x(t))).
        \end{split}
        \label{eq:horizontal}
    \end{equation}
    On the other hand, the vertical component of \eqref{eq:sol} is
    \begin{equation}
        \begin{split}
            \left[\bm{\eta} \left( \dv{x}{t} \right)\right]^{\#} &= \left[ \bm{\eta} \left(\mathrm{T}_{\widetilde{\gamma}(t)} \Phi_{g(t)} \left(\mathrm{T}_{g(t)} L_{g(t)^{-1}} \left( \dv{g}{t} \right) \right)^{\#} (\widetilde{\gamma}(t))\right)\right]^{\#} \\
            &= \left[ \mathrm{Ad}_{g(t)^{-1}} \mathrm{T}_{g(t)} L_{g(t)^{-1}} \left( \dv{g}{t} \right) \right]^{\#} = \left[ \mathrm{Ad}_{g(t)^{-1}} \bm{\eta}(X_{t}(\widetilde{\gamma}(t)))\right]^{\#}\\
            & = \left[ \bm{\eta} \left( \mathrm{T}_{\widetilde{\gamma}(t)} \Phi_{g(t)} (X_{t}(\widetilde{\gamma}(t))) \right)\right]^{\#} = \left[ \bm{\eta} (X_{t}(x(t)))\right]^{\#},
        \end{split}
        \label{eq:vertical}
    \end{equation}
    where we have used that, since $X_{t}$ is $G$-invariant for all $t \in \mathbb{R}$, then $[\bm{\eta}(X_{t})]^{\#}$ is also $G$-invariant for every $t \in \mathbb{R}$. From \eqref{eq:horizontal} and \eqref{eq:vertical} we see that $\dv{x}{t} = X_{t}(x(t))$ for all $t \in \mathbb{R}$, meaning that $x(t)$ is a solution of $X$. 
\end{proof}

Finally, we stress that these ideas can be extended to study broader classes of ODEs systems through Lie systems. For example, the following system on $\mathbb{R}^{2}_{x > 0}$
\begin{equation}
\dv{x}{t} =  n xy , \qquad \dv{y}{t} = -y^{2} + a_{1}(t) y + a_{2}(t),
    \label{eq:Gambier}
\end{equation}
is not a Lie system for linearly independent functions $ a_{1}(t), a_{2}(t) \in C^{\infty}(\mathbb{R})$ and $n$ a non-zero integer, and corresponds to a particular case of the so-called {\it Gambier} ($\mathrm{G}27$) {\it equation} \cite{Carinena2013,Gambier1910}. However, $Y := x \pdv{x}$ is a Lie symmetry of \eqref{eq:Gambier} inducing an $\mathbb{R}_{+}$-principal action on $\mathbb{R}^{2}_{x > 0}$.  Via the associated principal bundle $\pi: (x, y) \mapsto y$, system   \eqref{eq:Gambier} projects onto a Riccati equation 
\begin{equation}
    \dv{y}{t} = - y^{2} + a_{1}(t) y + a_{2}(t),
    \label{eq:Riccatti}
\end{equation}
which is a Lie system. Using the principal connection form $\alpha:= x^{-1} \dd x$,  solutions of  \eqref{eq:Gambier} can be reconstructed from those of \eqref{eq:Riccatti}, following the approach of Theorem~\ref{thm:reconstruction}.

A similar scenario arises in the following system on $\mathbb{R}{+} \times \mathbb{S}^{1}$ related to the Hopf bifurcation \cite[pp. 252--253]{Strogatz2018}
\begin{equation}
\dv{r}{t} = a(t) r + r^{3}, \qquad \dv{\theta}{t} = \omega(t) + \delta(t) r.
    \label{eq:Hopf}
\end{equation}
This is not a Lie system for linearly independent $a(t), \omega(t), \delta(t) \in C^{\infty}(\mathbb{R})$.  Nevertheless, $Y := \pdv{\theta}$ is a Lie symmetry of \eqref{eq:Hopf}, leading to an $\mathbb{S}^{1}$-principal bundle $\pi: (r, \theta) \mapsto r$, through which \eqref{eq:Hopf} projects onto the Lie system on $\mathbb{R}_{+}$ given by 
$$\dv{r}{t} = a(t) r + r^{3}.$$
 The reconstruction problem can be addressed via the principal connection form $\alpha := \dd \theta$, following the ansatz of  Theorem~\ref{thm:reconstruction}.

 The study of such systems deserves further analysis in future work.

\section*{Acknowledgements}

\phantomsection
\addcontentsline{toc}{section}{\\[-20pt]Acknowledgements}
\small
The author thanks  Tomasz Sobczak and Javier de Lucas for fruitful discussions. The author acknowledges financial support from the IDUB program of the University of Warsaw within the call I.1.5 with project number PSP: 501-D111-20-0001150, a fellowship (grant C15/23) funded by the Universidad Complutense de Madrid and Banco de Santander, as well as the support of Agencia Estatal de Investigaci\'on (Spain) under the grant PID2023-148373NB-I00 funded by MCIN/AEI/10.13039/501100011033/FEDER,UE and the contribution of RED2022-134301-T funded by MCIN/AEI/10.13039/501100011033 (Spain).

 \end{document}